\documentclass[a4paper,12pt]{amsart}
\usepackage[utf8x]{inputenc}
\usepackage[T1]{fontenc}
\usepackage[a4paper,includeheadfoot,margin=2.54cm]{geometry}
\usepackage{stmaryrd}
\usepackage{amsmath}
\usepackage{amssymb}
\usepackage{amsthm}
\usepackage{mathrsfs}
\usepackage[all]{xy}
\usepackage{xcolor}
\usepackage{enumitem}
\setenumerate[1]{itemsep=0pt,partopsep=0pt,parsep=\parskip,topsep=5pt}
\setitemize[1]{itemsep=0pt,partopsep=0pt,parsep=\parskip,topsep=5pt}
\setdescription{itemsep=0pt,partopsep=0pt,parsep=\parskip,topsep=5pt}
\usepackage{times}
\usepackage{mathabx}
\usepackage{amsrefs}
\usepackage[pdftex]{graphicx}
\usepackage[pdftex,
colorlinks=true,
linkcolor=purple,
citecolor=blue,
hyperindex,
plainpages=false,
bookmarks=true,
bookmarksopen=true,
bookmarksnumbered]{hyperref}
\theoremstyle{plain}
\newtheorem{theorem}{Theorem}[section]
\newtheorem*{theorem*}{Theorem}
\newtheorem{corollary}[theorem]{Corollary}
\newtheorem{lemma}[theorem]{Lemma}

\theoremstyle{definition}
\newtheorem{definition}[theorem]{Definition}

\theoremstyle{example}
\newtheorem{example}[theorem]{Example}

\theoremstyle{note}

\theoremstyle{remark}

\numberwithin{equation}{section}

\newcommand{\sd}{\mathbin{\triangle}}
\newcommand{\wc}{\rightharpoonup} 
\newcommand{\grass}[2]{\mathbf{G}(#1,#2)}
\newcommand{\Var}{\mathbf{V}}     
\newcommand{\RVar}{\mathbf{RV}}   
\newcommand{\IVar}{\mathbf{IV}}   
\newcommand{\var}{\mathbf{v}}     
\newcommand{\oball}{\mathbf{U}}
\newcommand{\cball}{\mathbf{B}}
\newcommand{\HM}{\mathcal{H}}
 
\newcommand{\mass}{\mathsf{M}} 
\newcommand{\fn}{\mathsf{W}} 
\newcommand{\RC}{\mathscr{R}}

\newcommand{\PC}{\mathscr{P}} 
\newcommand{\FC}{\mathscr{F}} 
 
\DeclareMathOperator{\VarTan}{VarTan}   
\DeclareMathOperator{\spt}{spt}
\DeclareMathOperator{\id}{id}
\DeclareMathOperator{\Tan}{Tan}
\DeclareMathOperator{\Lip}{Lip}
\DeclareMathOperator{\dist}{dist}
\DeclareMathOperator{\diam}{diam}
\DeclareMathOperator{\Vol}{Vol}
\DeclareMathOperator{\ap}{ap} 
\DeclareMathOperator{\mesh}{mesh}%

\newcommand{\mr}{\mathop{\vrule height 1.6ex depth 0pt width
0.13ex\vrule height 0.13ex depth 0pt width 1.3ex}\nolimits}
\newcommand{\ora}[1]{{\scriptstyle\xrightarrow{#1}}}
\begin{document}
\title{The weak convergence of varifolds generated 
	by rectifiable flat $G$-chains}
\author{Chunyan Liu, Yangqin Fang, \and Ning Zhang}
\address{CHUNYAN LIU, SCHOOL OF MATHEMATICS AND STATISTICS,
 HUAZHONG UNIVERSITY OF SCIENCE
	AND TECHNOLOGY, 430074, WUHAN, P.R. CHINA}
\email{chunyanliu@hust.edu.cn}
\address{YANGQIN FANG, SCHOOL OF MATHEMATICS AND STATISTICS,
HUAZHONG UNIVERSITY OF SCIENCE
	AND TECHNOLOGY, 430074, WUHAN, P.R. CHINA}
\email{yangqinfang@hust.edu.cn}
\address{NING ZHANG, SCHOOL OF MATHEMATICS AND STATISTICS,
HUAZHONG UNIVERSITY OF SCIENCE
	AND TECHNOLOGY, 430074, WUHAN, P.R. CHINA}
\email{nzhang2@hust.edu.cn}

\renewcommand{\thefootnote}{\fnsymbol{footnote}} 
\renewcommand{\thefootnote}{\arabic{footnote}} 
\date{}
\maketitle
\begin{abstract}
In the present paper, we prove that the convergence 
of rectifiable chains in flat norm implies the 
weak convergence of associated varifolds
if the limit flat chain is rectifiable and the mass 
converges to the mass of limit chain.
\end{abstract}
\section{Introduction} 
In 1960, Federer and Fleming \cite{FF:1960} initially established integral currents to solve
the Plateau's problem, and Fleming \cite{Fleming:1966} extended this theory to flat chains with coefficients in an abelian
group $G$. Setting $G$ a complete normed abelian
group, Fleming \cite{Fleming:1966} obtained that every finite mass
flat $G$-chain in $\mathbb{R}^n$ is rectifiable in case of the coefficient group $G$ is finite. Later, White \cite{White:1999:Acta,
White:1999:Ann} generalized this result to the case of
a coefficient group without nonconstant continuous path of finite length. On the other hand, varifolds was initiated  by Almgren
\cite{Almgren:1965} and extensively developed by W. Allard \cite{Allard:1972} as an alternative notion of surface (an orientation is not needed). The purpose of this paper is to discuss the connection between rectifiable chains and associated varifolds. 

Note that Convergence of flat chains means
convergence with respect to the flat norm, while convergence of varifolds is
weak convergence as Radon measures. Recently, a lot of attention has been attracted to constructing a bridge between these two convergences.
In 2009, White \cite{White:2009} considered the flat convergence of flat $Z_2$-chains associated to integral varifolds, and illustrated that the weak convergence of integral varifolds implied the flat
chains convergence in case of the sequence of integral varifolds satisfying 
the assumptions of Allard's compactness theorem and flat convergence of
boundary flat chains. In general, these two kinds of 
convergence cannot be derived from each other, see Example \ref{ex:1}. 

In this paper, we associate every rectifiable $G$-chain to a rectifiable
varifold, see Definition \ref{def:vf}. Thus, we think about the inverse of White's results \cite{White:2009} with the coefficient group $G$ being a
complete normed abelian group, that is, the weak convergence of rectifiable chains 
deduce the weak convergence of the associated 
rectifiable varifolds under appropriate assumptions. In particular, the main proof in the article is inspired by Fleming's work (see Lemma (8.1) in \cite{Fleming:1966}). However Fleming's technique is not appropriate for our case, so we develop new methods to deal with
it, see Lemma \ref{1.2}.

The paper is organized as follows.
Firstly, we show the implication in case of
polyhedral chains holds when the limit flat chain is
rectifiable and the mass of the sequence tending to the mass of the limit chain, see Theorem \ref{1.4}. Secondly, as an application of Theorem
\ref{1.4}, we get an inequality for the mass of flat chains, see Lemma \ref{1.5}. Thus, we prove that the implication remains true in case of rectifiable chains by
Lemma
\ref{1.2} and Lemma \ref{1.5}, see the Theorem \ref{thm:mainthm}.
Finally, we find the reverse statement is false under our conditions, see Example \ref{ex:1} and
\ref{ex:2}. For readers, we give a detailed proof of Theorem (6.1) and
Theorem (6.5) that appeared in the lecture notes of White \cite{White:2014}, see Lemma \ref{1.1} and \ref{1.3}. Our main results of the paper are as follows.
\begin{theorem}\label{1.4}
	Let $P_m\in\PC_d(\mathbb{R}^n;G)$ and $S\in\RC_d(\mathbb{R}^n;G)$. If $P_m$ converges to $S$ in flat norm and $\mass(P_m)\to\mass(S)$, then $\var(P_m)\wc\var(S)$ as varifolds.
\end{theorem}
\begin{theorem}\label{thm:mainthm}
Let $S_m,S\in\RC_d(\mathbb{R}^n;G)$ be flat chains of finite mass. If
$S_m$ converges to $S$ in flat norm and $\mass(S_m)\to\mass(S)$, then $\var(S_m)\wc\var(S)$ as varifolds.
\end{theorem}
For any compact set $X\subseteq \mathbb{R}^n$, we define $\FC_d(X;G)=\{S\in
\FC_d(\mathbb{R}^n;G):\spt S\subseteq X\}$ and $\RC_d(X;G)=\{S\in
\RC_d(\mathbb{R}^n;G):\spt S\subseteq X\}$. The following corollary directly follows from Theorem \ref{thm:mainthm}.
\begin{corollary}
	Let $X\subseteq \mathbb{R}^n$ be a compact set, and let $T\in \FC_{d-1}(X;G)$ be
	a flat chain. If $\{S_m\}\subseteq \RC_d(X;G)$ is a sequence of rectifiable
	chains such that $\partial S_m=T$, 
	\[
	\mass(S_m)\to \inf\{\mass(S):\partial
	S=T, S\in \RC_d(X;G)\}.
	\]
	and $S_m\ora\fn S$, then $\var(S_m)\to \var(S)$.
\end{corollary}
\section{Definitions and Preliminaries}
We collect some definitions and basic results that we use throughout the paper. For further facts we refer the reader to 
the books by Royden \cite{royden2017real}, Whitney \cite{Whitney:1957}, Federer \cite{Federer:1969} and the article by  Fleming \cite{Fleming:1966}.

The space of Radon measures  $\mathcal{M}(\mathbb{R}^n)$ is equipped with a {\it weak topology}, that is, $\mu_m\wc
\mu$ if and only if $\mu_m(\varphi)\to \mu(\varphi)$ for 
any $\varphi\in C_c(\mathbb{R}^n,\mathbb{R})$. 
Then the weak topology on unit ball in $\mathcal{M}(\mathbb{R}^n)$ is metrizable, see \cite[Section
15.4]{royden2017real}.

Let $d$ be a positive integer.
For any set $E\subseteq \mathbb{R}^n$, the  
{\it $d$-dimensional Hausdorff measure}
$\HM^d(E)$ is defined by 
\[
\HM^d(E)=\lim\limits_{\delta\to0}\inf
\left\{\sum\diam(U_{i})^d:E \subseteq\bigcup
U_{i}, U_i\subseteq\mathbb{R}^n, \diam(U_i)\leq\delta\right\}.
\]
For any Radon measure $\mu$ on $\mathbb{R}^n$, the {\it $d$-density} of $\mu$ at 
$a\in\mathbb{R}^n$ is defined by
\[
\Theta^d(\mu,a)=\lim\limits_{r\to0}
\frac{\mu(\cball(a,r))}{\omega_dr^d}.
\]
Here, $\omega_d$ is the $d$-dimensional
Hausdorff measure of the unit ball in $\mathbb{R}^d$.

A {\it normed abelian group} is an abelian group $G$ 
equipped with a norm
$\left|\cdot\right|:G\to[0,+\infty)$ 
satisfying
\begin{enumerate}
\item $\left|-g\right|=\left|g\right|$,
\item $\left|g+h\right|\leq\left|g\right|
+\left|h\right|$,
\item $\left|g\right|=0\text{ if and only if }g=0$.
\end{enumerate}
A normed abelian group $G$ is {\it complete} if it is complete 
with respect to the metric induced by the norm.

The group of 
{\it polyhedral chains} of dimension $d$ 
in $\mathbb{R}^n$, with coefficients in $G$, denoted by  $\PC_d(\mathbb{R}^n;G)$, is a collection of elements consist of
$\sum_{i=1}^n g_i\Delta_i$ with $g_i\in G$ 
and $\Delta_i$
are polyhedra of dimension $d$. The {\it mass} of $P$ 
is defined by
\[
	\mass(P)=\inf\left\{\sum\limits_{i=1}^{n}\left|g_i\right|
	\HM^d(\Delta_i):P=\sum_{i=1}^n g_i\Delta_i\right\}
\]
The  {\it Whitney flat norm} on $\PC_d(\mathbb{R}^n;G)$ 
is defined by 
\[
\fn(P)=\inf\left\{\mass(Q)+\mass(R):P=Q+\partial
R,Q\in\PC_d(\mathbb{R}^n;G),R\in\PC_{d+1}
(\mathbb{R}^n;G)\right\}.
\]
The group of {\it flat chains} of dimension $d$, $\FC_d(\mathbb{R}^n;G)$, is the completion of 
$\PC_d(\mathbb{R}^n;G)$ with respect to the Whitney flat norm $\fn$. For any $S\in\FC_d(\mathbb{R}^n;G)$, we denote by $P_i\ora\fn S$ a sequence of polyhedral chains $\{P_i\}$ converging to $S$ in Whitney flat norm
$\fn$. Then the mass of $S$ is defined by 
\[
\mass(S)=\inf\left\{\liminf\limits_{i\to\infty}
\mass(P_i):P_i\ora\fn S,P_i\in\PC_d(\mathbb{R}^n;G)
\right\}.
\]

A flat chain $S$ is
supported by a closed set $X$ if for every open set $U$
containing $X$, there is a sequence of polyhedral chains 
$\left\{P_i\right\}$ tending to $S$ with cells of each $P_i$ contained in $U$. The {\it support of $S$}, denoted by $\spt S$, is the smallest closed set $X$ supporting $S$, if it exits.

Let $\mathcal{K}$ be a $d$-simplicial complex in $\mathbb{R}^n$. Then the {\it mesh} of $\mathcal{K}$ is the maximal diameter of all simplexes in $\mathcal{K}$, i.e.
\[
	\mesh(\mathcal{K})=\max\left\{\diam(\sigma):\sigma
	\text{ is a simplex of the complex }\mathcal{K}\right\},
\]
while the {\it fullness} $\kappa(\sigma)$ of $d$-simplex $\sigma$ is defined by 
\[
	\kappa(\sigma)=\Vol(\sigma)/\diam(\sigma)^d.
\]

Let $\sigma=p_0\cdots p_d$ be a simplex, with its vertices given in 
the order shown; we construct its {\it standard subdivision} $\mathfrak{S}\sigma$ as follows. Set
\[
	p_{ij}=\frac{1}{2}(p_i+p_j), i\leq j,
\]
in particular, $p_{ii}=p_{i}$. These are the vertices of $\mathfrak{S}\sigma$. Define a 
partial ordering among these vertices by setting
\[
p_{ij}\leq p_{kl}, \text{ if }k\leq i,j\leq l.
\]
The simplexes of $\mathfrak{S}\sigma$ are all those formed from the
$p_{ij}$ in increasing order. Clearly the interiors of
these simplexes are disjoint; it is not hard to see that they actually form a simplical complex, which is a subdivision of $\sigma$.

If $\mathcal{K}$ is a simplicial complex, and we order its vertices in some fixed fashion, then $\mathfrak{S}_1\mathcal{K}$ may be 
formed by subdividing each of its simplexes as above. Now each simplex of $\mathfrak{S}_1\mathcal{K}$ has its vertices ordered, and 
we may subdivide again, forming $\mathfrak{S}_2\mathcal{K}$, etc. Then we form the sequence of standard subdivisions of $\mathcal{K}$,
relative to the given order of the vertices. Thus, there is a positive number
$\eta=\eta(\mathcal{K})>0$, such that $\kappa(\tau)\ge\eta$ for all simplexes
$\tau\in\mathfrak{S}_m\mathcal{K}$ and $\mesh(\mathfrak{S}_m\mathcal{K})\to 0$, see \cite[p359]{Whitney:1957}.

For any Lipschitz mapping $f: X\to\mathbb{R}^n$, $X\subseteq \mathbb{R}^n$, we
denote by $\ap J_d f$ the $d$-approximately Jacobian of $f$, see
\cite[Theorem 3.2.22]{Federer:1969} for details.

Let $\mathcal{K}$ be a $d$-simplicial complex. If $p_1,p_2,\cdots $ are the vertices of $\mathcal{K}$, then each simplex of $\mathcal{K}$ is of the form $p_{\lambda_{0}}\cdots p_{\lambda_{d}}$. Each point of $\mathcal{K}$ can be written uniquely in the form 
\[
p=\sum u_i(p)p_i, u_i(p)\ge0, \sum u_i(p)=1,
\]
with the conditon that if $p\in p_{\lambda_{0}}\cdots p_{\lambda_{d}}$, then $u_j(p)=0$ for $j\ne\lambda_0,\cdots,\lambda_{r}$. 
Let $f$ be a mapping of $\mathcal{K}$ into $\mathbb{R}^n$. For each point $p=\sum u_i(p)p_i$ of $\mathcal{K}$ as above, the corresponding {\it simplexwise affine approximation} $\bar{f}$ to $f$ is defined by 
\[
\bar{f}(p)=\bar{f}(\sum u_i(p)p_i)=\sum u_i(p)f(p_i),
\]
which is same as $f$ on each vertex of $\mathcal{K}$,
and is affine in each simplex of $\mathcal{K}$,
see \cite{Whitney:1957} for details.
\begin{lemma}\label{le:affapp}
Let $f$ be a Lipschitz mapping of $d$-polyhedron $\Delta\subseteq\mathbb{R}^n$ into $\mathbb{R}^n$. Then there is a
$\eta_0>0$ such that, for any $\varepsilon>0$, we can find $\xi>0$
satisfying that for any simplicial subdivision of $\Delta$ with 
the fullness greater than $\eta_0$ and mesh less than $\xi$,
\[
\int_{\Delta}\left|\ap J_d\bar{f}(x)-\ap J_df(x)\right|
d\HM^d(x)\leq\varepsilon,
\]
where $\bar{f}$ is the corresponding simplexwise affine approximation
to $f$ as above.
\end{lemma}
\begin{proof}
We take $\eta_0=\eta(\Delta)>0$ as above, and choose  
$\rho\in(0,1)$ such that
\[
\frac{6\rho\Lip(f)^d\HM^d(\Delta)}{(d-1)!\eta_0}\leq
\frac{\varepsilon}{3},~
\frac{(((d-1)!\eta_0)^d+1)\rho^d\Lip(f)^d\HM^d(\Delta)}
{((d-1)!\eta_0)^d}\leq\frac{\varepsilon}{3}.
\] 
Set 
\[
\gamma=((d-1)!\eta_0)^d\rho^d\varepsilon/3(((d-1)!\eta_0)^d+1)
\Lip(f)^d.
\]
Since $\Delta$ is a $d$-rectifiable set of $\mathbb{R}^n$,
$f$ is differentiable for $\HM^d$-
a.e. $x\in\Delta$. We may choose a compact set
$E\subseteq \Delta$ such that $Df$ is continuous in $E$ and $\HM^d(\Delta\setminus E)<\gamma$. Choose $\xi>0$ such that
\[
\left|Df(x)-Df(y)\right|\leq\rho\Lip(f)~~\text{ if }
x,y\in E,~\left|x-y\right|\leq\xi.
\]
Now we may suppose that $\Delta$ is a cell, and $\mathfrak{S}\Delta=
\cup_i\Delta_{i}$ is any simplicial subdivision 
satisfying the above conditions. Define $F=\Delta\setminus E$, 
reorder the sequence $\left\{\Delta_i\right\}$
so that there is a natural number $I$
such that $i\leq I$, $\HM^d(\Delta_i\cap F)\leq\rho^d\HM^d(\Delta_i)$.
Otherwise, the inequality fails for $i>I$. 
	
For each $d$-simplex $\Delta_i,~i\leq I$, if $x\in\Delta_i\cap E$, 
we set
\[
\Delta_i=p_{i0}\cdots p_{id},~v_{ij}=p_{ij}-p_{i0},~ 
f_i=f|_{\Delta_i},~\bar{f}_i=\bar{f}|_{\Delta_i},
\]
and
\[
u_{ij}=Df_i(x)(v_{ij}),~w_{ij}=f(p_{ij})-f(p_{i0})=D\bar{f}_i(x)(v_{ij}),
~\delta_i=\diam(\Delta_i).
\]
By Lemma 2a in \cite{Whitney:1957}, we see that
\[
\left|w_{ij}-u_{ij}\right|\leq6\rho\Lip(f)\delta_i,~
\left|u_{ij}\right|\leq\Lip(f)\delta_i,~
\left|w_{ij}\right|\leq\Lip(f)\delta_i.
\]
By \cite[(I,~12.17)]{Whitney:1957},
\[
\left|w_{i1}\wedge\cdots\wedge w_{id}-u_{i1}\wedge\cdots
\wedge u_{id}\right|<6d\rho\Lip(f)^d\delta_i^d.
\]
Set the unit $d$-vector
\[
\alpha_i=\frac{v_{i1}\wedge\cdots\wedge v_{id}}{|v_{i1}\wedge\cdots \wedge v_{id}|}
\]
By \cite[(III,~1.2)]{Whitney:1957},
\[\begin{aligned}
\big|\ap J_d\bar{f}_i(x)-\ap J_df_i(x)\big|&\leq\big|\wedge_dD\bar{f}_i(x)(\alpha_i)-\wedge_dDf_i(x)(\alpha_i)\big|\\&=\frac{\big|\wedge_dD\bar{f}_i(x)(v_{i1}\wedge\cdots\wedge v_{id})-\wedge_dDf_i(x)(v_{i1}\wedge\cdots\wedge v_{id})\big|}{|v_{i1}\wedge\cdots \wedge v_{id}|}
\\&=\frac
{\left|w_{i1}\wedge\cdots\wedge w_{id}-u_{i1}\wedge
\cdots\wedge u_{id}\right| }{d!\Vol(\Delta_i)}\\&\leq\frac{6d\rho\Lip(f)^d\delta_i^d}
{d!\Vol(\Delta_i)}=\frac{6\rho\Lip(f)^d}
{(d-1)!\kappa(\Delta_i)}.
\end{aligned}\]
Hence
\[\begin{aligned}
\int_{\Delta_i\cap E}\left|\ap J_d\bar{f}_i(x)-\ap 
J_df_i(x)\right|
d\HM^d(x)&\leq\frac{6\rho\Lip(f)^d\HM^d(\Delta_i)}{(d-1)!\eta_0}
\leq\frac{\varepsilon\HM^d(\Delta_i)}{3\HM^d(\Delta)}.
\end{aligned}\]
For any $x\in \Delta_i$, since $\ap J_d f_i(x)\leq \Lip(f)^d$ and 
\[
\ap J_d
\bar{f}_i(x)\leq \Lip(\bar{f})^d\leq \Lip(f)^d/[(d-1)!
\kappa(\Delta_i)]^d,
\]
Hence
\[
\int_{\Delta_i\cap F}\left|\ap J_d\bar{f}_i(x)-
\ap J_df_i(x)\right|d\HM^d(x)\leq
\frac{\Lip(f)^d(((d-1)!\eta_0)^d+1)\rho^d\HM^d(\Delta_i)}
{((d-1)!\eta_0)^d}
\leq\frac{\varepsilon\HM^d(\Delta_i)}
{3\HM^d(\Delta)}.
\]
In addition, set $\bar{\Delta}=\Delta_1\cup\cdots\cup\Delta_I$,
$\tilde{\Delta}=\Delta_{I+1}\cup\cdots$. 
If $i>I$, the definition of $\tilde{\Delta}$ gives 
\[
\rho^d\HM^d(\tilde{\Delta})<\sum_{i>I}
\HM^d(\Delta_i\cap F)\leq\HM^d(F)<\gamma.
\]
Thus
\[
\int_{\bar{\Delta}}\left|\ap J_d\bar{f}(x)-
\ap J_df(x)\right|d\HM^d(x)\leq \frac{2\varepsilon}{3},
\]
and
\[
\int_{\tilde{\Delta}}\left|\ap J_d\bar{f}(x)-
\ap J_df(x)\right|d\HM^d(x)\leq
\frac{(((d-1)!\eta_0)^d+1)\Lip(f)^d\gamma}{\rho^d[(d-1)!\eta_0]^d}=
\frac{\varepsilon}{3}.
\]
Hence
\[\int_{\Delta}\left|\ap J_d\bar{f}(x)-
\ap J_df(x)\right|d\HM^d(x)\leq\varepsilon.\]
\end{proof}

Let $P\in\PC_d(\mathbb{R}^n;G)$ be a polyhedral $d$-chain, and let $f$ be a Lipschitz mapping of $E=\spt P$ into $\mathbb{R}^n$. We define the corresponding Lipschitz mapping $f_{\sharp}P$ as follows. 
Let $\left\{\mathfrak{S}_kE\right\}$ be a sequence of simplicial subdivision of $E$ as above, 
and let $\left\{f_k\right\}$ be the
corresponding simplexwise affine mappings of $E$ into
$\mathbb{R}^n$. The limit of $\left\{f_{k\sharp}P\right\}$ exists and unique, see \cite[p296]{Whitney:1957} for details. Set
\[
f_{\sharp}P=\lim\limits_{k\to\infty}f_{k\sharp}P.
\]

Let $f$ be a Lipschitz mapping of $\mathbb{R}^n$ into $\mathbb{R}^n$. Then $f$ induces a chain map $f_{\sharp}$ of $\FC_d(\mathbb{R}^n;G)$ into $\FC_d(\mathbb{R}^n;G)$ as follows.
Let $S$ be a flat $d$-chain of $\FC_d(\mathbb{R}^n;G)$, and let
$\left\{P_m\right\}$ be a sequence of polyhedral $d$-chains 
 such that $P_m\ora\fn S$. We see that 
$\{f_{\sharp}P_m\}$ is a Cauchy sequence in the flat 
norm by the fact 
\[
	\fn(f_{\sharp}(P_i-P_j))\leq\max
	\left(\Lip(f)^d,\Lip(f)^{d+1}\right)\fn(P_i-P_j);
\]
we shall define
\[
	f_{\sharp}S=\lim\limits_{m\to\infty}f_{\sharp}P_m.
\]

To show that the limit unique, consider another sequence $\left\{Q_i\right\}$ of polyhedral $d$-chains  
which has the above property, thus
\[
\fn(f_{\sharp}(Q_m-P_m))\leq\max\left(\Lip(f)^d,\Lip(f)^{d+1}\right)
\fn(Q_m-P_m)\to 0.
\]

A flat chain $S\in\FC_d(\mathbb{R}^n;G)$ is called 
rectifiable if for each
$\varepsilon>0$ there exists a polyhedral $d$-chain 
$P\in \PC_d(\mathbb{R}^n;G)$ 
and a Lipschitz mapping 
$f:\mathbb{R}^n\to\mathbb{R}^n$ such that 
\[
\mass(S-f_{\sharp}P)<\varepsilon.
\]
We denote by $\RC_d(\mathbb{R}^n;G)$ the collection of 
all rectifiable $d$-chains,
and by $\mathscr{M}_d(\mathbb{R}^n;G)$ the collection of 
flat $d$-chains with finite mass. 

Let $P=\sum_{i=1}^{n}g_i\Delta_i$ be a polyhedral $d$-chain,
and let $I$ be an open interval. We define 
$P\mr I$ by the portion of $P$ in $I$, i.e. 
\[
P\mr I=\sum_{i=1}^{n}g_i(\Delta_i\cap I).
\]

For any polyhedral chain $P\in \PC_d(\mathbb{R}^n;G)$, there is
a Radon measure $\mu_P$ associated to the chain $P$, which is given by
$\mu_P(I)=\mass(P\mr I)$ for any $d$-dimensional interval $I$.

\begin{lemma}
For any flat chain $S\in\mathscr{M}_d(\mathbb{R}^n;G)$, if
$\{P_m\},\{Q_m\}\subseteq \PC_d(\mathbb{R}^n;G)$ satisfying that
$P_m\ora{\fn} S$, $Q_m\ora{\fn}S$, $\mass(P_m)\to \mass(S)$,
$\mass(Q_m)\to \mass(S)$, $\mu_{P_m}\wc \mu$ and $\mu_{Q_m}\wc
\nu$, then $\mu=\nu$. 
\end{lemma}
\begin{proof}
Let $\mathscr{A}$ be the collection of all open intervals $I\subseteq
\mathbb{R}^n$ such that $\mu(\partial I)=\nu(\partial I)=0$ and  
\[
\sum \Big(\fn(P_m\mr I-P_{m+1}\mr I)
+\fn(Q_m\mr I-Q_{m+1}\mr I)+\fn(P_m\mr I-Q_m
\mr I)\Big)<\infty,
\]
where $\partial I$ is the boundary of set $I$. We claim that $\mu(I)=\nu(I)$ for every
open interval $I\in\mathscr{A}$. Since
\[
\mu(I)\leq\liminf\limits_{m\to\infty}\mu_{P_m}(I)\leq
\limsup\limits_{m\to\infty}\mu_{P_m}(\bar{I})\leq\mu(\bar{I}),
\]
and $\mu(\partial I)=0$, we find that
\[
	\mu(I)=\lim\limits_{m\to\infty}\mu_{P_m}(I) \text{ and
	}\mu(\mathbb{R}^n\setminus I)=
	\lim\limits_{m\to\infty}\mu_{P_m}(\mathbb{R}^n\setminus I).
\]
For each $I\in\mathscr{A}$, we see that $P_m\mr I$ and 
$Q_m\mr I$ tend to
a same flat chain, saying that $T\in\FC_d(\mathbb{R}^n;G)$, then
\[
S-T=\lim\limits_{m\to\infty}
(P_m-P_m\mr I)=\lim\limits_{m\to\infty}(Q_m-Q_m\mr I).
\]
Since mass is lower semicontinuous,
\[\begin{aligned}
\mass(T)+\mass(S-T)&\leq\liminf\limits_{m\to\infty}
\mass(P_m\mr I)+\liminf\limits_{m\to\infty}
\mass(P_m\mr (\mathbb{R}^n\setminus I))\\&=\lim\limits_{m\to\infty}\mass(P_m)=\mass(S)
\leq\mass(T)+\mass(S-T).
\end{aligned}\]
Thus 
\[
\mass(T)=\lim\limits_{m\to\infty}
\mass(P_m\mr I)=\lim\limits_{m\to\infty}\mu_{P_m}(I)=\mu(I).
\]
Similarly,
\[
\mass(T)= \lim\limits_{m\to\infty}\mu_{Q_m}(I)=\nu(I).
\]
For each open interval $I$, by Lemma 2.1 in \cite{Fleming:1966} there exists a sequence of open intervals $\{I_j\}$ in $\mathscr{A}$ such that 
$I_j\subseteq I_{j+1}$ and $I=\cup I_j$, Hence
\[
\mu(I)=\lim\limits_{j\to\infty}\mu(I_j)=\lim\limits_{j\to\infty}\nu(I_j)=\nu
(I).
\]
\end{proof}
\begin{definition}
Let $S\in\mathscr{M}_d(\mathbb{R}^n;G)$ be a flat $d$-chain
with finite mass. Take a sequence $\left\{P_m\right\}$ of polyhedral $d$-chains  
such that $P_m\ora\fn S$, $\mass(P_m)\to\mass(S)$ and
$\mu_{P_m}$ tends weakly to a limit, then we define $\mu_S$ to be
the limit Radon measure.
\end{definition} 
By the above lemma, the Radon measure $\mu_S$ is well defined,
since the limit measure does not depend on the sequence $\{P_m\}$.

Let $S\in\mathscr{M}_d(\mathbb{R}^n;G)$ be any flat $d$-chain with
finite mass, and let $\left\{P_i\right\}$ be a sequence of 
polyhedral $d$-chains with $P_i\ora\fn S$ and $\mass(P_i)\to\mass(S)$. We denote by
$\mathscr{B}_{\{P_i\}}$
the collection of all intervals $I$ such that $\mu_S(\partial I)=0$ and $\sum_i
\fn(P_i\mr I-P_{i+1}\mr I)< \infty$. It is clear that $\{P_i\mr I\}$ converges to
a flat chain in flat norm when $I\in \mathscr{B}_{\{P_i\}}$.

\begin{lemma}
For any $S\in\mathscr{M}_d(\mathbb{R}^n;G)$ and Borel set $X$, if $\{P_i\}$ and $\{Q_i\}$ 
are two sequence of polyhedral chains satsfying that $P_i\ora{\fn} S$,
$Q_i\ora{\fn} S$, $\mass(P_i)\to \mass(S)$, 
$\mass(Q_i)\to \mass(S)$ and $\sum\Big(
\fn(P_i-P_{i+1})+\fn(Q_i-Q_{i+1})+\fn(P_i-Q_i)\Big)<\infty$, 
Then for any sequence 
$\{X_j\},\{Y_j\}$ of finite disjoint union of open intervals in
$\mathscr{B}_{\{P_i\}}$ and $\mathscr{B}_{\{Q_i\}}$ satisfying that
$\lim\limits_{j\to\infty}\mu_S(X_j\sd X)=0$ and 
$\lim\limits_{j\to\infty}\mu_S(Y_j\sd X)=0$ respectively, 
both sequence of flat chains
$\big\{\lim\limits_{i\to \infty}(P_i\mr X_j)\big\}$ and
$\big\{\lim\limits_{i\to \infty}(Q_i\mr Y_j)\big\}$ converge 
to a same flat chain in flat norm.
\end{lemma}
\begin{proof}
Since $X_j\in \mathscr{B}_{\{P_i\}}$, we find that 
$\lim_{i\to \infty}(P_i\mr X_j)$ exits.
\begin{equation}\label{eq:001}
\begin{aligned}
\mass\left(\lim_{i\to \infty}(P_i\mr X_{k})-\lim_{i\to \infty}(P_i\mr
X_{l})\right)&\leq \liminf_{i\to \infty}\mass(P_i\mr X_k-P_i\mr X_l)\\
&\leq\liminf_{i\to \infty}\big(\mu_{P_i}(X_k\setminus
X_{l})+\mu_{P_i}(X_l\setminus X_{k})\big) \\
&=\mu_{S}(X_{k}\sd X_{l})\leq \mu_S(X_k\sd X)+\mu_S(X_l\sd X).
\end{aligned}
\end{equation}
thus $\{\lim_{i\to \infty}(P_i\mr X_{j})\}$ is a Cauchy sequence 
in mass, it therefore converges to a flat chain in flat norm, saying $T$. 
We will show that the limit flat chain $T$ does not depend on the
choice of sequence $\{X_j\}$. Indeed, if we choose another 
sequence $\{Z_j\}$ of finite disjoint union of
open intervals in $\mathscr{B}_{\{P_i\}}$ such that 
$\mu_S(Z_j\sd X)\to 0$, then similar to \eqref{eq:001},
\[
\mass\left(\lim_{i\to \infty}(P_i\mr X_k)-\lim_{i\to \infty}(P_i\mr
Z_k)\right)\leq \mu_S(X_k\sd Z_k)\leq \mu_S(X_k \sd X)+\mu_S(Z_k\sd X).
\]
	
By a similar argument, $\{\lim_{i\to \infty}(Q_i\mr Y_j)\}$ is also converges in mass to 
a flat chain $R$ which does not depend on the choice of $\{Y_j\}$. 
	
By Lemma 2.1 in \cite{Fleming:1966}, there exists a sequence 
$\{U_j\}$ such that $U_j=\sqcup_{1\leq k\leq
j_0} I_{j,k}$, $I_{j,k}\in \mathscr{B}_{\{P_i\}}\cap
\mathscr{B}_{\{Q_i\}}$, $\sum_{i\geq 1} \fn \left( (P_i-Q_i)\mr
I_{j,k}\right)<\infty$ and $\mu_S(U_j \sd X)\to 0$. Thus
\[
\sum_{i=1}^{\infty}\fn(P_i\mr U_{j}-Q_i\mr U_{j} )\leq
\sum_{i=1}^{\infty}\sum_{k=1}^{j_0}\fn((P_i-Q_i)\mr I_{j,k})
=\sum_{k=1}^{j_0}\sum_{i=1}^{\infty}\fn((P_i-Q_i)\mr I_{j,k})<\infty.
\]
Hence 
\[
\lim_{i\to \infty}(P_i\mr U_j)=\lim_{i\to \infty}(Q_i\mr U_j),
\]
and 
\[
\lim_{j\to \infty}\lim_{i\to \infty}(P_i\mr X_j)=
\lim_{j\to \infty}\lim_{i\to\infty}(P_i\mr U_j)=
\lim_{j\to \infty}\lim_{i\to \infty}(Q_i\mr U_j)=\lim_{j\to
\infty}\lim_{i\to \infty}(Q_i\mr Y_j).
\]
\end{proof}
\begin{definition}
Let $S\in\mathscr{M}_d(\mathbb{R}^n;G)$ be a flat $d$-chain with 
finite mass, and let $X$ be a Borel set. Take a sequence of polyhedral $d$-chains $\{P_i\}$ 
such that $P_i\ora{\fn}S$, $\mass(P_i)\to \mass(S)$ and
$\sum\fn(P_i-P_{i+1})<\infty$, then we define
\[
S\mr X=\lim_{j\to \infty}\lim_{i\to \infty}P_i\mr X_j.
\]
for any sequence of $\{X_j\}$ of finite disjoint
union of open intervals in $\mathscr{B}_{\{P_i\}}$ satisfying that
$\mu_S(X_j\sd X)\to 0$, 
\end{definition}

Let $\grass{n}{d}$ be the Grassmann manifold which 
consists of $d$ dimensional
subspace of $\mathbb{R}^n$ and equipped with the metric 
\[
\dist(T_1,T_2)=\left\|(T_{1})_{\natural}-(T_2)_{\natural}\right\|,
\] 
where $T_{\natural}$ is the orthogonal projection 
onto the plane
$T$ and $\left\|.\right\|$ is the operator norm. 
A Radon measure
$V$ on the product space 
$\mathbb{R}^n\times \grass{n}{d}$ is called a
$d$-dimensional varifold in $\mathbb{R}^n$, and we write 
$\Var_d(\mathbb{R}^n)$ the
collection of all such varifolds in $\mathbb{R}^n$. 
We say a varifold $V\in
\Var_d(\mathbb{R}^n)$ is rectifiable if there exist a 
$d$-rectifiable set
$M$ and a nonnegative function 
$\theta:M\to\mathbb{R^{+}}$ for $\HM^d$-a.e.
$x\in M$ such that 
\[
V=(\theta\HM^d\mr M)\otimes\delta_{\Tan(M,x)},
\] 
where $\Tan(M,x)$ is the approximate tangent $d$-plane of $M$ at the 
point $x$ which exists $\HM^d$-a.e. $x\in M$. If the function
$\theta$ is integer-valued for $\HM^d$-a.e. $x\in M$, 
we say $V$ is a
$d$-integral varifold in $\mathbb{R}^n$. We denote by 
$\RVar_d(\mathbb{R}^n)$ and
$\IVar_d(\mathbb{R}^n)$ the collections of all 
$d$-dimensional rectifiable 
varifolds in $\mathbb{R}^n$ and all $d$-dimensional 
integral varifolds in
$\mathbb{R}^n$, respectively.
For any rectifiable chain 
$S\in\RC_d(\mathbb{R}^n;G)$,  
there exists a $d$-rectifiable set $M$ such that 
\[
\mu_S\left(\mathbb{R}^n\setminus M\right)=0.
\]
\begin{definition} \label{def:vf}
Let $S\in \RC_d(\mathbb{R}^n;G)$ be a rectifiable chain, 
and let $M$ be a $d$-rectifiable set such
that $\mu_S(\mathbb{R}^n\setminus M)=0$. 
Then the associated rectifiable varifold $\var(S)$ is defined by 
\[
\var(S)=\mu_S\otimes\delta_{\Tan{(M,x)}}.
\]
\end{definition}
For any $V\in\Var_d(\mathbb{R}^n)$, the weight 
$\left\|V\right\|$ is defined by the requirement that 
\[
\left\|V\right\|(A)=V(A\times G(n,d))
\] 
for any Borel set $A\subseteq\mathbb{R}^n$.

Let $f:\mathbb{R}^n\to\mathbb{R}^n$ be a $C^1$ 
mapping. Then there is a induced mapping between 
varifolds $f_{\sharp}:V_d(\mathbb{R}^n)
\to V_d(\mathbb{R}^n)$ given by
\[
f_{\sharp}V(\varphi)=\displaystyle
\int_{\mathbb{R}^n\times\grass{n}{d}}
\varphi(f(x),Df(x)T)\left\|
\wedge_dDf(x)\circ T_{\natural}\right\|dV(x,T)
\] 
for any $V\in \Var_d(\mathbb{R}^n)$ and 
$\varphi\in C_c(\mathbb{R}^n\times
\grass{n}{d},\mathbb{R})$, see \cite{Allard:1972}. 
For any $a\in \mathbb{R}^n$ and 
$V\in \Var_d(\mathbb{R}^n)$, a varifold 
$C\in\Var_d(\mathbb{R}^n)$ is called a varifold 
tangent of $V$ at $a$, if there is
a sequence of positive numbers $\{r_i\}$ with 
$\lim\limits_{i\to\infty} r_i=0$ such that 
\[
C=\lim\limits_{i\to\infty}(T_{a,r_i})_{\sharp}V,
\]
where $T_{a,r_i}$ is the mapping defined by 
$T_{a,r_i}(x)=r_i^{-1}(x-a)$. 
We denote by $\VarTan(V,a)$ the
collection of all varifold tangents of $V$ at the 
point $a$.
\section{The weak convergence of varifolds}
Let $V, V_m\in\Var_d(\mathbb{R}^n)$ be varifolds in 
$\mathbb{R}^n$. We say that $V_m$ weak converges to 
$V$, denote by $V_m\wc V$, if
\[
V_m(\varphi)\to V(\varphi), \forall \varphi
\in C_c(\mathbb{R}^n\times\grass{n}{d},\mathbb{R}).
\]
\begin{lemma}\label{1.1}
If $S_m$, $S\in\RC_d(\mathbb{R}^n;G)$, $S_m\ora\fn S$ 
and $\mass(S_m)\to\mass(S)$, then $\mu_{S_m}\wc\mu_S$.
\end{lemma}
\begin{proof}
By taking subsequence it suffices to prove this for
sequence $\left\{S_m\right\}$ such that 
\[
\sum_{m}\fn(S_m-S_{m+1})<\infty.
\]
By the definition of
$\mu_{S_m}$, we can find a sequence of polyhedral chains 
$\left\{P_{m,\ell}\right\}$ such 
that $P_{m,\ell}\ora\fn S_m$ and $\mass(P_{m,\ell})\to\mass(S_m)$ as $\ell\to \infty$.
By taking subsequence we may assume that 
\[
\sum_{\ell}\fn(P_{m,\ell}-P_{m,\ell+1})<\infty,
\]
then $\mu_{P_{m,\ell}}\wc\mu_{S_m}$. Let metric $\rho$ be a 
metrization of the weak topology on a large ball in
$\mathcal{M}(\mathbb{R}^n)$,
which contains $\mu_S$ and the sequence $\{\mu_{S_m}\}$. 
Thus $\rho(\mu_{P_{m,\ell}},\mu_{S_m})\to0$. Since $S_m\ora\fn S$ and
$\mass(S_m)\to\mass(S)$, then we can choose a natural number $\ell_m$ 
such that
\[
\fn(P_{m,\ell_m}-S_{m})<2^{-m}, 
\left|\mass(P_{m, \ell_m})-\mass(S_m)\right|<2^{-m}\text{ and }
\rho(\mu_{P_{m,\ell_m}},\mu_{S_m})<2^{-m}.
\] 
Thus 
$P_{m,\ell_m}\ora\fn S$, $\mass(P_{m,\ell_m})\to\mass(S)$ and
$\sum_{m}\fn(P_{m,\ell_m}-P_{m,\ell_{m+1}})<\infty$. Hence 
$\mu_{P_{m,\ell_m}}\wc\mu_S$, therefore $\mu_{S_m}\wc\mu_S$.
\end{proof}

\begin{lemma}\label{1.2}
Let $S\in\RC_d(\mathbb{R}^n;G)$ be a rectifiable 
chain, and let $M$ be a $d$-rectifiable set such that 
$\mu_S(\mathbb{R}^n\setminus M)=0$. Then for 
$\mu_S$-almost every $x\in M$,
\[
\lim\limits_{r\to0}
\frac{\mass(\pi_{x\sharp}(S\mr\cball(x,r)))}
{\mass(S\mr\cball(x,r))}=1,
\]
where $\pi_{x}$ is the orthogonal projection of 
$\mathbb{R}^n$ onto approxiamate tangent $d$-plane $\Tan(M,x)$.
\end{lemma}
\begin{proof}
Since $M$ is a $d$-rectifiable set and $\mu_S$ is a 
$d$-rectifiable measure, then the approximate tangent $d$-plane $\Tan(M,x)$ 
exists for $\mu_S$-a.e.
$x\in M$, and $\Theta^d(\mu_S,x)$ exists and positive 
for $\mu_S$-a.e. $x\in M$. So we may assume that 
$\Tan(M,x)$ exists and $\Theta^d(\mu_S,x)>0$ for all 
$x\in M$.

If $\mass(\partial S)<\infty$,
by contradiction we assume that there is a positive 
$\mu_S$ set $X\subseteq M$ such that 
\[
\liminf_{r\to0}\frac{\mass(\pi_{x\sharp}
(S\mr\cball(x,r)))}{\mass(S\mr\cball(x,r))}<1.
\]
Then there exists $\delta>0$ such that 
\[
X_{\delta}=\left\{x\in X:\liminf_{r\to0}
\frac{\mass(\pi_{x\sharp}(S\mr\cball(x,r)))}
{\mass(S\mr\cball(x,r))}<1-\delta\right\}
\]
has $\mu_S$ measure positive.  For any
$0<\varepsilon<\min\{\delta/(4^{d+2}(d+1)),1/3\}$, 
\[
\lim_{r\to0}r^{-d}\mu_S(M\cap\cball(x,r)\setminus 
\mathcal{C}(x,T_x,\varepsilon))=0,
\]
where $T_x=\Tan(M,x)$ and
$\mathcal{C}(x,T_x,\varepsilon)=
\{y\in\mathbb{R}^n:\dist(y,T_x)\leq
\varepsilon|y-x|\}$.
Then the collection $\mathscr{C}$ of closed balls 
$\cball(x,r)$ such that
$x\in X_{\delta}$, $0<r\leq \varepsilon$, 
\[
\mass(\pi_{x\sharp}(S\mr\cball(x,r)))\leq(1-\delta)
\mass(S\mr\cball(x,r)),
\]
\[
\left|\frac{\mu_S(\cball(x,\rho))}{\omega_d\rho^d}
-\Theta^d(\mu_S,x)\right|\leq\varepsilon
\Theta^d(\mu_S,x),\forall 0<\rho\leq r
\]
and 
\[
\rho^{-d}\mu_S(M\cap\cball(x,\rho)\setminus 
\mathcal{C}(x,T_x,\varepsilon))
\leq\varepsilon\omega_d\Theta^d(\mu_S,x), 
\forall 0<\rho\leq r,
\]
is a Vitali covering of $X_{\delta}$. By the Vitali 
covering theorem, there is countable disjoint 
subcollection of such balls $\cball(x_i,r_i)$ of the 
radius $r_i\leq\varepsilon$ covering the set 
$X_\delta$. So there exists a natural number $N$ such 
that 
\[
\mu_{S}(X_{\delta}\setminus\bigcup_{i=1}^{N}
\cball(x_i,r_i))<\varepsilon.
\] 
We put 
\[
A_i=(\Tan(M,x_i)+\cball(0,\varepsilon
r_i))\cap\cball(x_i,(1-2\varepsilon)r_i).
\]
then 
\[
A_i+\cball(0,\varepsilon r_i)\subseteq\cball(x_i,r_i).
\] 
Let $\varphi_{i}:\mathbb{R}^n\to \mathbb{R}^n$ be the 
mapping given by 
\[
\varphi_i(x)=x-\eta(\dist(x,A_i)/r_i)
(T_i)_{\natural}^\perp(x-x_i), \forall
x\in\mathbb{R}^n,
\]
where $T_i=\Tan(M,x_i)$, 
$(T_i)_{\natural}^\perp:\mathbb{R}^n\to T_i^\perp$ 
is an orthogonal projection and 
\[
\eta(t)=\begin{cases}
1,& t\leq 0;\\
1-t/\varepsilon,&0\leq t\leq \varepsilon;\\
0,&t>\varepsilon.
\end{cases}
\]
Thus $\Lip(\eta)=1/\varepsilon$,
$\varphi_i(\cball(x_i,r_i))\subseteq\cball(x_i,r_i)$, 
$\varphi_i(x)=x$ for
$x\in \mathbb{R}^n\setminus \cball(x_i,r_i)$.
For each $x,y\in\mathbb{R}^n$, if
$\min\left\{\dist(x,A_i),\dist(y,A_i)\right\}
\ge\varepsilon r_i$,
\[
\left|\varphi_i(x)-\varphi_i(y)\right|=
\left|x-y\right|.
\] 
Otherwise 
$\min\left\{\dist(x,A_i),\dist(y,A_i)\right\}
<\varepsilon r_i$, we assume
that $\dist(y,A_i)<\varepsilon r_i$. Then
\[
\left|(T_i)_{\natural}^{\perp}(y-x_i)\right|
=\dist(y,T_i)\leq 2\varepsilon r_i.
\]
Thus
\[\begin{aligned}
\left|\varphi_{i}(x)-\varphi_{i}(y)\right|&=
\left|x-y+\eta(\dist(y,A_i)/r_i)
(T_i)_{\natural}^{\perp}(y-x_i)-\eta(\dist(x,A_i)/r_i)
(T_i)_{\natural}^{\perp}(x-x_i)\right|\\
&\leq\left|x-y\right|+\left|\eta(\dist(y,A_i)/r_i)
(T_i)_{\natural}^{\perp}
(y-x_i)-\eta(\dist(x,A_i)/r_i)
(T_i)_{\natural}^{\perp}(x-x_i)\right|\\
&\leq\left|x-y\right|+\left|(\eta(\dist(y,A_i)/r_i)-
\eta(\dist(x,A_i)/r_i))(T_i)_{\natural}^{\perp}
(y-x_i)\right|\\&+
\left|\eta(\dist(x,A_i)/r_i)
(T_i)_{\natural}^{\perp}(x-y)\right|\\
&\leq\left|x-y\right|+\frac{1}{r_i}
\Lip(\eta)\left|\dist(y,A_i)-
\dist(x,A_i)\right|\left|(T_i)_{\natural}^{\perp}
(y-x_i)\right|+	\left|x-y\right|\\
&\leq2\left|x-y\right|+\frac{1}
{\varepsilon r_i}2\varepsilon r_i\left|x-y\right|\\
&\leq4\left|x-y\right|.
\end{aligned}\] 
Therefore $\Lip(\varphi_i)\leq 4$. 
We put 
$\varphi_{\varepsilon}=\varphi_1\circ\varphi_2
\cdots\circ\varphi_N$
and 
\[
h_{\varepsilon}(u,x)=(1-u)x+u
\varphi_{\varepsilon}(x), 0\leq u\leq 1.
\]
We claim that $\Lip(\varphi_{\varepsilon})\leq4$. For any $x,y\in\mathbb{R}^n$, if $x,y\in\cball(x_i,r_i), i=1,\cdots, N$, since the closed balls $\cball(x_i,r_i)$ are disjoint,
$\varphi_{\varepsilon}(x)=\varphi_i(x), \varphi_{\varepsilon}(y)=\varphi_i(y)$, then 
$|\varphi_{\varepsilon}(x)-\varphi_{\varepsilon}(y)|\leq4|x-y|$; if
$x,y\in\mathbb{R}^n\setminus\bigcup_{i=1}^{N}\cball(x_i,r_i)$, $\varphi_{\varepsilon}(x)=x, \varphi_{\varepsilon}(y)=y$, then $|\varphi_{\varepsilon}(x)-\varphi_{\varepsilon}(y)|=|x-y|$; if
$x\in\cball(x_i,r_i),y\in\mathbb{R}^n\setminus\bigcup_{j=1}^N\cball(x_j,r_j)$, $\varphi_{\varepsilon}(x)=\varphi_i(x)$, $\varphi_{\varepsilon}(y)=y$, then
$|\varphi_{\varepsilon}(x)-\varphi_{\varepsilon}(y)|=|\varphi_i(x)-y|\leq4|x-y|$; if
$x\in\cball(x_i,r_i), y\in\cball(x_j,r_j), i\ne j$, since the closed balls are disjoint, we can find a point
$z=tx+(1-t)y\in\mathbb{R}^n\setminus\bigcup_{k=1}^N\oball(x_k,r_k)$, $0<t<1$ such that
\[
|\varphi_{\varepsilon}(x)-\varphi_{\varepsilon}(y)|\leq|\varphi_{\varepsilon}(x)-z|+|z-\varphi_{\varepsilon}(y)|\leq4(|x-z|+|z-y|)=4|x-y|.
\]
By (6.3) in \cite[p172]{Fleming:1966},
\[
\varphi_{\varepsilon\sharp}(S)-S=
\partial h_{\varepsilon\sharp}
(I\times S)+h_{\varepsilon\sharp}(I\times\partial S), I=[0,1].
\]
By (6.5) in \cite[p172]{Fleming:1966}, 
\begin{equation}\label{eq:1}
\mass(h_{\varepsilon\sharp}
(I\times S))\leq\left\|\varphi_{\varepsilon}-
\id\right\|_{\infty}\Lip(\varphi_{\varepsilon})^d
\mass(S),
\end{equation}
and
\begin{equation}\label{eq:2}
\mass(h_{\varepsilon\sharp}(I\times
\partial S))\leq\left\|\varphi_{\varepsilon}-
\id\right\|_{\infty}\Lip(\varphi_{\varepsilon})^{d-1}
\mass(\partial S).
\end{equation}
Since $\varphi_{\varepsilon}({\cball(x_i,r_i)})=\varphi_i({\cball(x_i,r_i)})\subseteq\cball(x_i,r_i),i=1,\cdots,N$, then if $x\in\bigcup_{i=1}^N\cball(x_i,r_i)$, 
$|\varphi_{\varepsilon}(x)-x|\leq2r_i\leq2\varepsilon$. Otherwise $|\varphi_{\varepsilon}(x)-x|=0$. Therefore $\|\varphi_{\varepsilon}-\id\|_{\infty}\leq2\varepsilon$.
From \eqref{eq:1} and \eqref{eq:2},
\[\begin{aligned}
\fn(\varphi_{\varepsilon\sharp}(S)-S)&
\leq\fn(h_{\varepsilon\sharp}
(I\times S))+\fn(h_{\varepsilon\sharp}
(I\times\partial S))\\
&\leq\mass(h_{\varepsilon\sharp}
(I\times S))+\mass(h_{\varepsilon\sharp}
(I\times\partial S))\\
&\leq\left\|\varphi_{\varepsilon}-\id\right\|_
{\infty}4^d(\mass(S)+\mass(\partial S))\\
&\leq2\varepsilon 4^d(\mass(S)+\mass(\partial S)).
\end{aligned}\]
Letting $\varepsilon\to0$,
$\varphi_{\varepsilon\sharp}(S)\ora\fn S$.  
By the lower semicontinuity of $\mass$, 
\[
\mass(S)\leq\lim\inf\limits\mass
(\varphi_{\varepsilon\sharp}S).
\] 
Set
\[
B_i=\cball(x_i,r_i),B=\bigcup\limits_{i=1}^{N}B_i,A=
\bigcup\limits_{i=1}^{N}A_i,
B_i'=\cball(x_i,(1-2\varepsilon)r_i),
C_i=\mathcal{C}(x_i,T_i,\varepsilon).
\]
But we get 
\[\begin{aligned}
\mass(\varphi_{\varepsilon\sharp}S)
&\leq\mass(\varphi_{\varepsilon\sharp}(S\mr A))
+\mass(\varphi_{\varepsilon\sharp}
(S\mr(B\setminus A)))+\mu_S(\mathbb{R}^n\setminus B)
\\&\leq\sum\limits_{i=1}^{N}\left(\mass((T_i)_
{\natural}(S\mr A_i))+
\Lip(\varphi_{\varepsilon})
\mass(S\mr(B_i\setminus A_i))\right)+
\mu_{S}(\mathbb{R}^n\setminus B)\\&
\leq\sum\limits_{i=1}^{N}((1-\delta)
\mass(S\mr B_i)+4^d(\mass(S\mr(B_i\setminus B_i'))+
\mass(S\mr(B_i'\setminus
C_i))))+\mu_S(\mathbb{R}^n\setminus B)\\&
\leq\mu_S(\mathbb{R}^n)-\sum\limits_{i=1}^{N}
(\delta\mass(S\mr B_i)-4^d(1+
\varepsilon-(1-2\varepsilon)^{d+1}
+\varepsilon(1-2\varepsilon)^{d})
\Theta^d(\mu_S,x_i)\omega_dr_i^d)\\&=
\mass(S)-\sum\limits_{i=1}^{N}(\delta-4^d
[1+\varepsilon-(1-2\varepsilon)^{d+1}+
\varepsilon(1-2\varepsilon)^{d}]/(1-\varepsilon))
\mass(S\mr B_i)\\&
\leq \mass(S)-3 \delta \mu_S(X_{\delta})/4.
\end{aligned}\] 
This contradicts lower semicontnuity 
of mass.
 
If $\mass(\partial S)=\infty$, for each $\varepsilon\in(0,1)$,
there exists a Lipschitz chain $R$ such that
$\mass(S-R)<\varepsilon^2$. 
Let $M'$ be a rectifiable set such
that $\mu_R(\mathbb{R}^n\setminus M')=0$. From the above proof,
we get that for $\mu_R$-almost every $x\in M'$, 
\[
\lim\limits_{r\to0}
\frac{\mass(\pi'_{x\sharp}(R\mr\cball(x,r)))}
{\mass(R\mr\cball(x,r))}=1,
\]
where $\pi_{x}'$ is an orthogonal projection of 
$\mathbb{R}^n$ onto approxiamate tangent $d$-plane $\Tan(M',x)$. Define
\[
Y_{\varepsilon}=\left\{x\in M:\liminf_{r\to0}
\frac{\mass(\pi_{x\sharp}(S\mr\cball(x,r)))}
{\mass(S\mr\cball(x,r))}<1-\varepsilon\right\}
\]
and
\[
X_{\varepsilon}=\left\{x\in M\cap M':\liminf_{r\to0}
\frac{\mass(\pi_{x\sharp}(S\mr\cball(x,r)))}
{\mass(S\mr\cball(x,r))}<1-\varepsilon\right\}.
\]
Then the collection $\mathscr{C}'$ of closed balls 
$\cball(x,r)$ such that $x\in X_{\varepsilon}$, 
$0<r<\varepsilon$,
\[
\mass(\pi_{x\sharp}(S\mr\cball(x,r)))\leq(1-\varepsilon)
\mass(S\mr\cball(x,r)),
\]
\[
\mass(\pi'_{x\sharp}(R\mr\cball(x,r)))>(1-\varepsilon/2)
\mass(R\mr\cball(x,r))
\]
is a Vitali convering of $X_{\varepsilon}$. By the Vitali
convering theorem, there is countable disjoint subcollection
of such balls $\cball(x_i,r_i)$ of the radius $r_i\leq\varepsilon$
convering the set $X_{\varepsilon}$. Set $B_i=\cball(x_i,r_i)$,
since $\pi_x=\pi_x'$ for $\mu_R$-almost every $x\in X_{\varepsilon}$, 
thus
\[\begin{aligned}
\mass(\pi'_{x\sharp}(R\mr B_i))&\leq\mass(\pi'_{x\sharp}((R-S)\mr B_i))
+\mass(\pi_{x\sharp}(S\mr B_i))\\&\leq\mass((R-S)\mr B_i)+
(1-\varepsilon)\mass(S\mr B_i)\\&\leq\mass((R-S)\mr B_i)+
(1-\varepsilon)[\mass(R\mr B_i)+\mass((S-R)\mr B_i)]\\&
\leq(1-\varepsilon)\mass(R\mr B_i)+2\mass((S-R)\mr B_i). 
\end{aligned}\]
Since $\mass(\pi'_{x\sharp}(R\mr B_i))>(1-\varepsilon/2)
\mass(R\mr B_i)$, then
\[
\mass(R\mr B_i)<\frac{4\mass\left((S-R)\mr B_i\right)}{\varepsilon}.
\]
Since 
\[\begin{aligned}
\mu_R\left(\bigcup\limits_iB_i\right)=\sum_i\mu_R(B_i)=
\sum_i\mass(R\mr B_i)\leq\frac{4\mass\left((S-R)
\mr\left(\bigcup_iB_i\right)\right)}{\varepsilon}\leq
\frac{4\mass(S-R)}{\varepsilon},
\end{aligned}\]
and
\[\begin{aligned}
\mass(S\mr(M\setminus M'))&\leq\mass((S-R)
\mr(M\setminus M'))+\mass(R\mr(M\setminus M'))\\&
\leq\mass(S-R)+\mass(R\mr(\mathbb{R}^n\setminus M'))
\leq\varepsilon^2.
\end{aligned}\]
Thus
\[\begin{aligned}
\mu_S(X_\varepsilon)&\leq\sum_i\mu_S(B_i)\leq
\sum_i(\mass(R\mr B_i)+\mass((S-R)\mr B_i))
\\&\leq\frac{4\mass(S-R)}{\varepsilon}+\mass(S-R)\leq4\varepsilon
+\varepsilon^2\leq5\varepsilon,
\end{aligned}\]
and
\[
\mu_S(Y_{\varepsilon}\setminus X_{\varepsilon})\leq
\mu_S(M\setminus M')=\mass(S\mr(M\setminus M'))\leq\varepsilon^2.
\]
Then 
\[
\mu_S(Y_{\varepsilon})=\mu_S(X_{\varepsilon})+
\mu_S(Y_{\varepsilon}\setminus X_{\varepsilon})\leq
5\varepsilon+\varepsilon^2\leq6\varepsilon.
\]
\end{proof}
\begin{lemma}\label{1.3}
Let $f:\mathbb{R}^n\to\mathbb{R}^n$ be a 
Lipschitz mapping. For any 
$P\in\PC_d(\mathbb{R}^n;G)$, by setting $E=\spt P$,
\[
\mass(f_{\sharp}P)\leq\displaystyle\int_E
\ap J_d(f|_{E})(x)d\mu_P(x).
\]
\end{lemma}
\begin{proof}
Since $f_{\sharp}P$ is a Lipschitz chain, then there
exists a sequence of standard subdivisions 
$\left\{\mathfrak{S}_kE\right\}$ and a sequence of
corresponding simplxwise affine mappings 
$\left\{f_k:f_k:E\to\mathbb{R}^n\right\}$ such that
$f_{\sharp}P=\lim\limits_{k\to\infty}f_{k\sharp}P$. 
We assume that $P=\sum_i g_i \Delta_i$ with $\HM^d(\Delta_i\cap \Delta_j)=0$
for $i\neq j$. Suppose that $ \mathfrak{S}_k \Delta_i=
\cup_{j}\Delta_{ij}^{(k)} $. Then
$f_{k\sharp}P=\sum_{i,j}g_if_k(\Delta_{ij}^{(k)})$.
By Lemma \ref{le:affapp}, 
for each $\varepsilon>0$, there exist
a natural number $N$ such that $k\ge N$, 
\[
\int_{E}\left|\ap J_d(f|_{E})(x)-
\ap J_df_k(x)\right|d\mu_{P}(x)<\varepsilon.
\] 
Then 
\[
	\begin{aligned}
\mass(f_{k\sharp}P)&\leq\sum_{i,j}|g_i|
\HM^d(f_k(\Delta_{ij}^{(k)}))=\sum_{i,j}|g_i|
\int_{\Delta_{ij}^{(k)}}\ap J_df_k(x)d\HM^d(x)\\&=
\int_{E}\ap J_df_k(x)d\mu_{P}(x)\leq\int_{E}\ap
J_d(f|_{E})(x)d\mu_P(x)+\varepsilon.
\end{aligned}\]
By the lower semicontinuity of $M$, letting 
$\varepsilon\to0$,
\[
\mass(f_{\sharp}P)\leq\int_{E}\ap J_d(f|_{E})(x)
d\mu_P(x).
\]
\end{proof}

\begin{proof}[Proof of Theorem \ref{1.4}]
Define $\var_m=\var(P_m),\mu_m=\mu_{P_m}$, since 
$\sup\limits_{m\to\infty}\var_m(\mathbb{R}^n
\times\grass{n}{d})=\sup\limits_{m\to\infty}
\mu_m(\mathbb{R}^n)<\infty$, by the compactness 
theorem of Radon theorem, there is a subsequence 
$\left\{\var_m'\right\}$ and a varifold
$V$ such that $\var_m'\wc V$. We need to prove 
$V=\var(S)$. For any 
$P_m, S\in\RC_d(\mathbb{R}^n;G)$, there exist a
$d$-rectifiable set $E_m$ and a $d$-rectifiable set $M$ such 
that $\mu_{P_m}(\mathbb{R}^n\setminus E_m)=0
\text{ and } \mu_S(\mathbb{R}^n\setminus M)=0$, 
respectively. Then there is a set $M'\subseteq M$ 
and $\mu_S(M')=0$ such that the approximate tangent $d$-plane $\Tan(M,x)$ exists 
and the density $\Theta^d(\mu_S,x)>0$ for all 
$x\in M\setminus M'$.

By the lemma \ref{1.1}, we know that 
$\lim\limits_{m\to\infty}\mu_{m}'=\mu_{S}$. 
Fixed $a\in M\setminus M'$,
we show that for each $C\in\VarTan(V,a)$, 
corresponding to a 
$\Theta^d(\mu_S,a)\var(\Tan(M,a))$ such that 
$C=\Theta^d(\mu_S,a)\var(\Tan(M,a))$. 
Define $T_a=\Tan(M,a)$, $T_x=\Tan(E_m',x),x\in E_m'$, 
then for each $\varphi\in C_c^{\infty}(\cball(0,1)
\times G(n,d),\mathbb{R})$,
\[\begin{aligned}
C(\varphi)&=\lim\limits_{i\to\infty}
(T_{a,r_i})_{\sharp}V(\varphi)=
\lim\limits_{i\to\infty}\lim\limits_{m\to\infty}
(T_{a,r_i})_{\sharp}\var_m'(\varphi)\\&=
\lim\limits_{i\to\infty}\lim\limits_{m\to\infty}
\displaystyle\int_{\cball(0,1)
\times \grass{n}{d}}\varphi(x,T)
d(T_{a,r_i})_{\sharp}\var_m'(x,T)\\&=
\lim\limits_{i\to\infty}\lim\limits_{m\to\infty}
\frac{1}{r_i^d}\displaystyle\int_{\cball(a,r_i)
\times \grass{n}{d}}\varphi
\left(r_i^{-1}(x-a),T\right)d\var_m'(x,T)\\&
=\lim\limits_{i\to\infty}\lim\limits_{m\to\infty}
\frac{1}{r_i^d}\displaystyle
\int_{E_m'\cap\cball(a,r_i)}
\varphi\left(r_i^{-1}(x-a),T_x\right)
d\mu_m'(x).
\end{aligned}\]
Similarly,
\[\begin{aligned}
\Theta^d(\mu_S,a)\var(T_a)(\varphi)&=
\Theta^d(\mu_S,a)\displaystyle
\int_{\cball(0,1)\times\grass{n}{d}}
\varphi(x,T)d\var(T_a)(x,T)\\&=
\Theta^d(\mu_S,a)\displaystyle
\int_{\cball(0,1)\cap T_a}\varphi(x,T_a)d\HM^d(x)\\&
=\Theta^d(\mu_S,a)\displaystyle\int_{\cball(0,1)
\cap T_a}\psi(x)d\HM^d(x)\\&=
\Theta^d(\mu_S,a)\HM^d\mr T_a(\psi),
\end{aligned}\]
where $\psi\in C_c^{\infty}(\cball(0,1),\mathbb{R})$ 
is defined by $\psi(x)=\varphi(x,T_a)$. 
In addition, since 
\[
\VarTan(\mu_S,a)=\Theta^d(\mu_S,a)\HM^d\mr T_a,
\]
then we can write
\[
\Theta^d(\mu_S,a)\HM^d\mr T_a=\lim\limits_{i\to\infty}
(T_{a,r_i})_{\sharp}\mu_S.
\] 
Thus
\[\begin{aligned}
\Theta^d(\mu_S,a)\var(T_a)(\varphi)&=
\lim\limits_{i\to\infty}(T_{a,r_i})_{\sharp}
\mu_S(\psi)=\lim\limits_{i\to\infty}
\lim\limits_{m\to\infty}
(T_{a,r_i})_{\sharp}\mu_m'(\psi)\\&=
\lim\limits_{i\to\infty}\lim\limits_{m\to\infty}
\frac{1}{r_i^d}\displaystyle\int_{\cball(a,r_i)}
\psi\left(r_i^{-1}(x-a)\right)
d\mu_m'(x)\\&=\lim\limits_{i\to\infty}
\lim\limits_{m\to\infty}\frac{1}{r_i^d}
\displaystyle\int_{E_m'\cap\cball(a,r_i)}
\varphi\left(r_i^{-1}(x-a),T_a\right)d\mu_m'(x).
\end{aligned}\]
Therefore
\[\begin{aligned}
\left|C(\varphi)-\Theta^d(\mu_{S},a)\var(T_a)
(\varphi)\right|&\leq\lim\limits_{i\to\infty}
\lim\limits_{m\to\infty}\frac{1}{r_i^d}
\displaystyle\int_{E_m'\cap\cball(a,r_i)}
\left|\varphi(r_i^{-1}(x-a),T_x)-
\varphi(r_i^{-1}(x-a),T_a)\right|
d\mu_m'(x)\\&\leq\limsup_{i\to\infty}
\limsup_{m\to\infty}\frac{1}{r_i^d}
\left\|D\varphi\right\|_{\infty}
\displaystyle\int_{E_m'\cap\cball(a,r_i)}
\left\|(T_{x})_{\natural}-(T_a)_{\natural}\right\|
d\mu_m'(x).
\end{aligned}\]
We claim that
\[
\limsup_{i\to\infty}\limsup_{m\to\infty}
\frac{1}{r_i^d}
\left\|D\varphi\right\|_{\infty}\displaystyle
\int_{E_m'\cap\cball(a,r_i)}
\left\|(T_{x})_{\natural}-(T_a)_{\natural}\right\|
d\mu_m'(x)=0.
\] 
Employing \cite[8.9(3)]{Allard:1972}, we get that
\[\left\|(T_{x})_{\natural}-(T_a)_{\natural}\right\|=
\left\|(T_a)_{\natural}^\perp\circ(T_{x})_{\natural}\right\|=
\sup\limits_{u\in T_x,\left|u\right|=1}
\left|(T_a)_{\natural}^\perp(u)\right|,
\] 
so that we can find $u_1\in T_x$ with $\left|u_1\right|=1$ such that
$\left\|(T_{x})_{\natural}-(T_a)_{\natural}\right\|=
\left|(T_a)_{\natural}^\perp(u_1)\right|$. Let $u_1,\cdots,u_d$ 
be an unit orthogonal basis of $T_x$, and let 
$p:E_n'\cap\cball(a,r_i)\to T_a$ be an orthogonal mapping defined by 
$p(x)=(T_a)_{\natural}(x)$, then
\[
\ap J_dp(x)=\left|(T_a)_{\natural}(u_1)\wedge\cdots
\wedge(T_a)_{\natural}(u_d)\right|\leq
\left|(T_a)_{\natural}(u_1)\right|=
\left(1-\left\|(T_{x})_{\natural}-(T_a)_{\natural}
\right\|^2\right)^{1/2}.
\]
Thus 
\[
\left\|(T_{x})_{\natural}-(T_a)_{\natural}\right\|^2
\leq1-\ap J_dp(x)^2\leq 2\left(1-\ap J_dp(x)\right).
\] 
Since $\mu_S(\partial\cball(a,r))=0$ for almost all $r$, we can assume that the choice of all $r_i$ above satisfying $\mu_S(\partial\cball(a,r_i))=0$, then
$\lim\limits_{m\to\infty}\mu_m'(\cball(a,r_i))=\mu_S(\cball(a,r_i))$.
By Lemma (4.2) in \cite{Fleming:1966},
\[
P_m'\mr\cball(x_i,r_i)\ora\fn S\mr\cball(x_i,r_i),
\]
so 
$p_{\sharp}(P_m'\mr\cball(a,r_i))\ora\fn p_{\sharp}(S\mr\cball(a,r_i))$. By the lower semicontinuity of $\mass$,
\[
\mass(p_{\sharp}(S\mr\cball(a,r_i))\leq\liminf_{n\to\infty}\mass(p_{\sharp}(P_m'\mr\cball(a,r_i)))
\]
By Lemma \ref{1.3},
\[
\displaystyle\int_{E_m'\cap\cball(a,r_i)}(1-\ap J_dp(x))d\mu_m'(x)\leq
\mass(P_m'\mr\cball(a,r_i))-\mass(p_{\sharp}(P_m'\mr\cball(a,r_i))).
\]  
Thus 
\[\begin{aligned}
&\limsup_{m\to\infty}\displaystyle\int_{E_m'\cap\cball(a,r_i)}(1-\ap J_dp(x))d\mu_m'(x)\\&\leq\limsup_{m\to\infty}\mass(P_m'\mr\cball(a,r_i))
-\liminf_{m\to\infty}\mass(p_{\sharp}(P_m'\mr\cball(a,r_i)))\\&
\leq\mass(S\mr\cball(a,r_i))-\mass(p_{\sharp}(S\mr\cball(a,r_i))).
\end{aligned}\]
By Lemma \ref{1.2},
\[
\lim\limits_{i\to\infty}\frac{1}{r_i^{d}}[\mass(S\mr\cball(a,r_i))-\mass(p_{\sharp}(S\mr\cball(a,r_i)))]=0.
\]
So
\[
\limsup_{i\to\infty}\limsup_{m\to\infty}\frac{1}{r_i^{d}}\displaystyle\int_{E_m'\cap\cball(a,r_i)}(1-\ap J_dp(x))d\mu_m'(x)=0.
\]
Since 
\[\begin{aligned}
\limsup_{i\to\infty}\limsup_{m\to\infty}\frac{\mu_m'(E_m'\cap\cball(a,r_i))}{{r_i^d}}&\leq\limsup_{i\to\infty}\limsup_{m\to\infty}\frac
{\mu_m'(\cball(a,r_i))}{r_i^d}\\&\leq\lim_{i\to\infty}\frac{\mu_S(\cball(a,r_i))}{r_i^d}=\omega_d\Theta^d(\mu_S,a)
\end{aligned}\]
By the H$\ddot{\text{o}}$lder's inequality,
\[\begin{aligned}
&\limsup_{i\to\infty}\limsup_{m\to\infty}\frac{1}{r_i^d}\displaystyle\int_{E_m'\cap\cball(a,r_i)}
\left\|(T_{x})_{\natural}-(T_a)_{\natural}\right\|
d\mu_m'(x)\\&\leq\limsup_{i\to\infty}\limsup_{m\to\infty}\sqrt{2}\Big(\frac{\mu_m'(E_m'\cap\cball(a,r_i))}{{r_i^d}}.\frac{1}{r_i^d}
\displaystyle\int_{E_m'\cap\cball(a,r_i)}
(1-\ap J_dp(x))d\mu_m'(x)\Big)^{\frac{1}{2}}\\&=0.
\end{aligned}\]
Thus
\[
C\mr(\cball(0,1)\times\grass{n}{d})=\Theta^d(\mu_S,a)\var(T_a)
\mr(\cball(0,1)\times\grass{n}{d}),
\]
but both $C$ and $\Theta^d(\mu_S,a)\var(T_a)$ are cones, 
so that $C=\Theta^d(\mu_S,a)\var(T_a)$, that is,  $\VarTan(V,a)=\VarTan(\var(S),a)$. Hence $V=\var(S)$.
\end{proof}
\begin{lemma}\label{1.5}
Let $f:\mathbb{R}^n\to\mathbb{R}^n$ be a $C^1$
mapping, and let $S\in\RC_d(\mathbb{R}^n;G)$ be a rectifiable flat chain.
Suppose that $M$ is a $d$-rectifiable set such that 
$\mu_S(\mathbb{R}^n\setminus M)=0$. Then
\begin{equation}\label{eq:coml}
\mass(f_{\sharp}S)\leq\int_{M}\ap J_d(f|_M)(x)d\mu_{S}(x).
\end{equation}
\end{lemma}
\begin{proof}
If $\spt S$ is compact, we can find a 
sequence $\left\{P_m\right\}$ of polyhedral $d$-chains  
such that 
\[	
P_m\ora\fn S,\mass(P_m)\to\mass(S),
\spt P_m\subseteq\spt S+\oball(0,1).
\] 
From Theorem \ref{1.4}, we get that $\var(P_m)\wc \var(S)$. Thus for any
$\varphi\in C_c(\mathbb{R}^n\times \grass{n}{d},\mathbb{R})$, 
\begin{equation}\label{eq:3}
	\lim_{m\to \infty}\var(P_m)(\varphi)=\var(S)(\varphi).
\end{equation} 
Since $f$ is a $C^1$ mapping, we have that $Df$ is continuous.
Let $\varphi:\mathbb{R}^n\times G(n,d)\to\mathbb{R}$ be the continuous
function given by
\[
\varphi(x,T)=\left\|\wedge_d Df(x)\circ T_{\natural}\right\|.
\] 
Set $E_m=\spt P_m$; $T_{m}(x)=\Tan(E_m,x),x\in E_m$; $T(x)=\Tan(M,x), x\in M$.
Since the support of $\var(P_m)$ is uniformly bounded, we can apply $\varphi$ to
\eqref{eq:3}, and we get that
\[
\lim\limits_{m\to\infty}\displaystyle\int_{E_m}
\left\|\wedge_d Df(x)\circ T_{m}(x)_{\natural}\right\|
d\mu_{P_m}(x)=\displaystyle\int_{M}\left\|\wedge_d
Df(x)\circ T(x)_{\natural}\right\|d\mu_S(x).
\] 
By Lemma \ref{1.3},
\[
\mass(f_{\sharp}P_m)\leq\int_{E_m}
\ap J_d(f|_{E_m})(x)d\mu_{P_m}(x).
\]
By the lower semicontinuity of $M$,
\[
	\begin{aligned}
\mass(f_{\sharp}S)&\leq\liminf_{m\to\infty}
\mass(f_{\sharp}P_m)\leq\liminf_{m\to\infty} \int_{E_m}
\ap J_d(f|_{E_m})(x)d\mu_{P_m}(x)\\
&= \lim_{m\to\infty}\int_{E_m}
\left\|\wedge_dDf(x)\circ T_{m}(x)_{\natural}\right\| d\mu_{P_m}(x)\\
&= \int_{M}\left\|\wedge_d
Df(x)\circ T(x)_{\natural}\right\|d\mu_S(x)=
\int_{M}\ap J_d(f|_M)(x)d\mu_S(x).
\end{aligned}
\]
If $\spt S$ is not compact, we take a sequence of positive real 
numbers $\{r_m\}$ such that $r_m\to \infty$ and
$\mu_S(\partial\cball(0,r_m))=0$. Set $S_m=S\mr\cball(0,r_m)$,
then $S_m$ converges in mass to $S$ and
$\mu_{S_m}(\mathbb{R}^n\setminus M)\to 0$. Thus \eqref{eq:coml} holds for
$S_m$, and we get that
\[
\mass(f_{\sharp}S)=\lim\limits_{m\to\infty}\mass(f_{\sharp}S_m)
\leq\lim\limits_{m\to\infty}\int_{M}\ap J_d(f|_M)(x)d\mu_{S_m}(x)
=\int_{M}\ap J_d(f|_M)(x)d\mu_S(x).
\]

\end{proof}
\begin{proof}[Proof of Theorem \ref{thm:mainthm}]
The proof is similar to Theorem \ref{1.4}, we see that the mapping $p$ in the proof of Theorem \ref{1.4}
is only an orthogonal mapping, so it is also $C^1$ mapping. Therefore we only need to replace Lemma \ref{1.3} with Lemma \ref{1.5}.
\end{proof}
\section{Examples}
From following example, we see that the reverse statement of Theorem
\ref{thm:mainthm} is untrue. Thus $\var(S_m)\wc\var(S)$,
$\partial S_m=\partial S=0$ and $\mass(S_m)\to \mass(S)$ do not imply $S_m\ora\fn S$.
\begin{example}\label{ex:1}
Let 
\[
\Omega_m= \bigcup_{1\leq k\leq m} \left\{z\in \mathbb{C}:\frac{2k-1}{m}\pi\leq
\arg z \leq \frac{2k}{m}\pi, 1\leq |z|\leq 1+\frac{1}{m^2}\right\}
\]
and let $S_m$ be the one-dimensional flat $\mathbb{Z}_2$-chain 
associated to set-theoretic boundary of $\Omega_m$. 
Since the area of $\Omega_m$ tends to zero, 
we get that $S_m$ tends to 0 in flat norm. That is, $S_m\ora{\fn} 0$
and $\partial S_m=0$, but $\var(S_m)\to \var(S)$, 
where $S$ is the one-dimensional flat $\mathbb{Z}_2$-chain 
associated to the unit circle.	
\end{example}
The following example shows that $\var(S_m)\wc \var(S)$ and
$S_m\ora{\fn} S$ do not imply $\mass(S_m)\to \mass(S)$ in general, 
if the support of the sequence
is not uniformly bounded.
\begin{example}\label{ex:2}
Let $Q_m=[m,m+1]\times [0,1/m]$ be a rectangle in $\mathbb{R}^{2}$ and
let $S_m$ be the one-dimensional flat $\mathbb{Z}_2$-chain 
associated to set-theoretic boundary of $Q_m$. Since the area 
of $Q_m$ tends to zero, we have that $S_m$ tends to 0 in flat norm.
It is not hard to see that $\var(S_m)\wc 0$. 
But $\mass(S_m)\to 2\neq 0$.
\end{example}
\bibliography{1}
\end{document}